\documentclass[10pt]{article}

\usepackage{amsmath}
\usepackage{amsfonts}
\usepackage{amsthm}
\usepackage{amssymb} 

\allowdisplaybreaks[3]

\theoremstyle{plain}
\newtheorem{theorem}{Theorem}[section]
\newtheorem{lemma}[theorem]{Lemma}

\theoremstyle{definition}

\numberwithin{equation}{section} 

\numberwithin{subconj}{conjecture} 

\newcommand\RR{\mathbb{R}}

\newcommand\NN{\mathbb{N}}

\title{A proof of the Veselov Conjecture for segments%
\footnote{Partially supported by PGC2018-096504-B-C31
(Minis\-te\-rio de Ciencia, Innovaci\'on y Universidades),
FQM-262 and US-1254600 (Jun\-ta de Anda\-lu\-c\'ia)
and Feder Funds (European Union).}}

\author{Antonio J. Dur\'{a}n\\
     \footnotesize
        \  Departamento de An\'{a}lisis Matem\'{a}tico.
       Universidad de Sevilla \\
       \footnotesize Apdo (P. O. BOX) 1160. 41080 Sevilla. Spain.
   duran@us.es \\
          \ \ }

\date{}

\begin{document}

\sloppy

\maketitle

\begin{abstract}
In this note, we prove Veselov's conjecture on the zeros of Wronskians whose entries are Hermite polynomials when the degrees of the polynomials are consecutive positive integres.


\end{abstract}

\section{Introduction}

Let $(H_n)_n$ be the sequence of Hermite polynomials, i.e., the polynomial eigenfunctions of the second order differential operator
$$
-\frac{d^2}{dx^2}+2x\frac{d}{dx},
$$
with the usual normalization $H_n(x)=2^nx^n+\cdots$, $n=0,1,2,\cdots$

From now on, $F$ will denote a finite set of positive integers. We will write $F=\{ f_1,\cdots , f_k\}$, with $f_i<f_{i+1}$.
We associate to $F$ the Wronskian
\begin{equation}\label{omh}
\Omega_F(x)=\det (H_{f_i}^{(j-1)}(x))_{i,j=1}^k.
\end{equation}
It is not difficult to see that for any set $F$, the Wronskian $\Omega_F$ is a polynomial of degree
\begin{equation}\label{degWf}
w_F=\sum_{f\in F}f-\binom{k}{2}.
\end{equation}

Because of its mathematical and physical interest, Wronskian determinants whose entries are Hermite polynomials and the properties of their zeros have long been considered in the literature. For instance, Karlin and Szeg\H{o} considered such determinants
in their celebrated paper \cite{KS}, devoted to extending Tur\'an's inequality for Legendre polynomials  to Hankel determinants whose entries are ultraspherical, Laguerre, and Hermite polynomials. Karlin and Szeg\H{o} studied the real zeros of the Wronskian determinants (\ref{omh}) when $F$ is the segment $F=\{n,n+1,n+2,\cdots, n+k-1\}$.
More precisely, they proved that if $k$ is even then $\Omega _F$ has no zeros in $\RR$, and if $k$ is odd then $\Omega _F$ has exactly $n$ simple zeros in~$\RR$ (notice that the degree of $\Omega_F$ is $kn$). In the case when $F$ is a segment, the zero set of $\Omega_F$  can be also interpreted as the pole set of
some rational solutions of the fourth Painlev\'e equation and has
a regular rectangle-like structure in the complex plane, as was pointed out by Clarkson \cite{Clar}.

Consider a sequence of eigenfunctions $(\phi_n)_n$, $n\ge 0$, for a Schr\"{o}dinger operator of the form $T=-d^2/dx^2+U$, and form the Wronskian
$$
\Omega_F^T(x)=\det (\phi_{f_i}^{(j-1)}(x))_{i,j=1}^k.
$$
The Hermite case corresponds to $U(x)=x^2$. In the context of these Schr\"{o}dinger operators, Krein and Adler, independently \cite{Kr,Ad}, characterized those sets $F$ for which the Wronskian $\Omega^T_F$ does not vanish in the real line. This leads to the concept of admissible sets: a finite set of positive integer $F$ is said to be admissible if
$$
\prod_{f\in F}(x-f)\ge 0, x\in \NN.
$$
Assuming mild conditions on the function $U$, Krein and Alder proved that $F$ is admissible if and only if the Wronskian $\Omega ^T_F(x)$ does not vanish in the real line (see also \cite{duch} and \cite{GUGM}).

The zero set of the Wronskians (\ref{omh}) is also related with the Calogero-Moser problem and the log-gas in a harmonic field; see \cite{Ves}.
For an study of Wronskians of Hermite polynomials from a combinatorial view point see \cite{BDS}.

Oblomkov also considered Wronskian whose entries are Hermite polynomials to study certain class of Schr\"{o}dinger operators $T=-d^2/dx^2+U$ with rational potentials $U$ growing as $x^2$ at infinity \cite{Obl}. In this context, Alexander Veselov made the following conjecture in the 1990s (it is explicitly quoted in \cite{FHV}):

\bigskip

\noindent
\textit{Veselov's conjecture.}

\textit{For every finite set $F$ of positive integers, the zeros of $\Omega _F(z)$ (\ref{omh}) are simple, except possible for $z=0$.}

\medskip

The purpose of this note is to prove Veselov's conjecture when $F$ is a segment.  The finite set $F$ of positive integers is called a segment if its elements are consecutive integers.

\begin{theorem}\label{thp}
If $F$ is a segment then the zeros of $\Omega _F(z)$ are simple.
\end{theorem}

Wronskian  determinants whose entries are Hermite and the other families of classical orthogonal polynomials are nowadays receiving increasing interest because of their role in the construction of the so-called exceptional  polynomials. In particular, we will use the differential properties of the exceptional Hermite polynomials in the proof of the Theorem \ref{thp}.
Exceptional orthogonal polynomials $p_n$, $p_n$ of degree $n$ and $n\in \NN\setminus X$ with $X\not =\emptyset$ finite, are complete orthogonal polynomial systems with respect to a positive measure which in addition are eigenfunctions of a second order differential operator. A family of exceptional polynomials have gaps in their degrees (the finite set $X$ above), and hence these families are not covered by the hypotheses of Bochner's classification theorem (see \cite{B}).

Exceptional polynomials extend the classical families of Hermite, Laguerre and Jacobi, and they have opened a new field of very active research, especially in
mathematical physics (see, for instance,
\cite{duch,dume,duhj,GUKM2} (where the adjective \textrm{exceptional} for this topic was introduced), \cite{GUGM,GFGM,OS0,OS3,Qu} and the references therein). For each family of exceptional polynomials there is a quantum-mechanical potential whose spectrum and
eigenfunctions can be calculated exactly. The applications of exceptional polynomials in
mathematical physics are very numerous, and include the description of mixed states \cite{GUGM2,OS0}, shape-invariant potentials \cite{Qu}, scattering amplitudes \cite{YKM}, superintegrable systems \cite{PTM,MaQu}, diffusion and stochastic processes \cite{Ho}, entropy and information theory \cite{DuRo}, solutions of the Dirac
equation \cite{ScRo} or finite-gap potentials \cite{FHV}, all of them developed in recent years.

\section{Proof of the Veselov conjecture for segments}

We will need some previous results.

The first result is the Sylvester's determinant identity (for the proof and a more general formulation of the Sylvester's identity see \cite{Gant}, p. 32).

\begin{lemma}\label{lemS}
For a square matrix $M=(m_{i,j})_{i,j=1}^k$,  and for each $1\le i, j\le k$, denote by $M_i^j$ the square matrix that results from $M$ by deleting the $i$-th row and the $j$-th column. Similarly, for $1\le i, j, p,q\le k$ denote by $M_{i,j}^{p,q}$ the square matrix that results from $M$ by deleting the $i$-th and $j$-th rows and the $p$-th and $q$-th columns.
The Sylvester's determinant identity establishes that for $i_0,i_1, j_0,j_1$ with $1\le i_0<i_1\le k$ and $1\le j_0<j_1\le k$, then
$$
\det(M) \det(M_{i_0,i_1}^{j_0,j_1}) = \det(M_{i_0}^{j_0})\det(M_{i_1}^{j_1}) - \det(M_{i_0}^{j_1}) \det(M_{i_1}^{j_0}).
$$
\end{lemma}

The second result is a very interesting invariance of the Wronskians (\ref{omh}). Indeed, consider the following mapping $I$ defined in the set $\Upsilon$  formed by all finite sets of positive integers:
\begin{align}\nonumber
I&:\Upsilon \to \Upsilon \\\label{dinv}
I(F)=\{1,2,&\cdots, \max F\}\setminus \{\max F-f,f\in F\}.
\end{align}
It is not difficult to see that $I$ is an involution.

We then have
\begin{equation}\label{hip}
c_F\Omega_{F}(x)=c_{I(F)}i^{w_F}\Omega_{I(F)}(-ix)
\end{equation}
where $w_F$ is given by (\ref{degWf}) and $c_F$ is the constant (independent of $x$) defined by:
$$
c_F=\frac{1}{2^{\binom{k}{2}}\prod_{f\in F}f!}.
$$
The identity (\ref{hip}) was reported in \cite{FHV} (without proof) and also conjectured in \cite{duch}.
It was proved independently in \cite{cyd} and \cite{GUGM3}.

We also need some more results concerning the exceptional Hermite polynomials.
The exceptional Hermite family associated to the finite set $F$ of positive integers is
defined by means of the Wronskian
\begin{equation}\label{defhex}
H_n^F(x)=
  \begin{vmatrix}
    H_{n-u_F}(x)&H_{n-u_F}'(x)&\cdots &H_{n-u_F}^{(k)}(x) \\
    H_{f_1}(x)&H_{f_1}'(x)&\cdots &H_{f_1}^{(k)}(x)\\
    \vdots&\vdots&\ddots &\vdots\\
    H_{f_k}(x)&H_{f_k}'(x)&\cdots &H_{f_k}^{(k)}(x)
  \end{vmatrix}.
\end{equation}
If we define
\begin{align}\label{defuf}
u_F&=\sum_{f\in F}f-\binom{k+1}{2},\\\label{defsf}
\sigma _F&=\{u_F,u_F+1,u_F+2,\cdots \}\setminus \{u_F+f,f\in F\},
\end{align}
we have that for $n\in \sigma _F$, $H_n^F$ is a polynomial of degree $n$.  But for $n\not \in \sigma_F$ the determinant (\ref{defhex}) vanishes and then $H_n^{F}=0$.

The exceptional Hermite polynomials are eigenfunctions of the second order differential operator
\begin{equation}\label{xsodo}
D_F=-\partial ^2+h_1(x)\partial+h_0(x),
\end{equation}
where
\begin{align}\label{lum}
h_1(x)&=2\left(x+\frac{\Omega_F'(x)}{\Omega_F(x)}\right),\\\nonumber
h_0(x)&=2\left(k+u_F-x\frac{\Omega_F'(x)}{\Omega_F(x)}\right)-\frac{\Omega_F''(x)}{\Omega_F(x)},
\end{align}
and
\begin{equation}\label{odom}
D_F(H_n^F)=2nH_n^F,\quad n\in \sigma_F
\end{equation}
(see (5.11) and (5.12) of \cite{duch}).

The most interesting case appear when $F$ is admissible. Then, the exceptional Hermite polynomials
are orthogonal in the real line with respect to the positive weight
\begin{equation}\label{xhw}
w(x)=\frac{e^{-x^2}}{\Omega_F^2(x)}.
\end{equation}

We use the following notation to manage segments: if $p,q\ge 1$
\begin{equation}\label{notseg}
S(p,q)=\{p,p+1,\cdots, p+q-1\}.
\end{equation}
For $q=0$, we set $S(p,q)=\emptyset$.

Theorem \ref{thp} is the second part in the following Theorem (the proof follows a similar approach as the corresponding for Casoratian determinants whose entries are
Charlier polynomials: see Lemma 5.4 of \cite{dua}).

\begin{theorem}\label{segment}
Let $F$ be the segment $F=S(p,q)$, with $q\ge 1$. Then
\begin{enumerate}
\item[(A.1)] the polynomials $\Omega_{S(p,q)}$ and $\Omega_{S(p+1,q-1)}$ have no common zeros.
\item[(A.2)] The zeros of the polynomial $\Omega_{S(p,q)}$ are simple.
\item[(A.3)] The polynomials $\Omega_{S(p,q)}$ and $\Omega_{S(p+1,q)}$ have no common zeros.
\end{enumerate}
\end{theorem}

\begin{proof}
We prove the three properties by induction on $q$.

For $q=1$, we have $S(p,1)=\{p\}$ and $S(p+1,0)=\emptyset$. For $F=\emptyset$, we set $\Omega_F(x)=1$. Hence, property (A.1) for $q=1$ establishes that the $p$-th Hermite polynomial $H_p$ and $1$ have no common zeros, which it is trivially true.

Property (A.2) for $q=1$ establishes the well-known fact  that the zeros of $H_p(z)$ are simple.

Property (A.3) for $q=1$ establishes that $H_p$ and $H_{p+1}$ have no common zeros, which it is again a well-known property.

Assume now that properties (A.1), (A.2) and (A.3) hold for $q$.

We first prove that the polynomials $\Omega_{S(p,q+1)}$ and $\Omega_{S(p+1,q)}$ have no common zeros ((A.1) for $q+1$).

Indeed, the Sylvester's identity of Lemma \ref{lemS} for $M=\Omega_{S(p,q+1)}$ gives
\begin{align*}
\Omega_{S(p+1,q-1)}(x)\Omega_{S(p,q+1)}(x)=\Omega_{S(p+1,q)}&(x)\Omega_{S(p,q)}'(x)
\\&\qquad-\Omega_{S(p+1,q)}'(x)\Omega_{S(p,q)}(x).
\end{align*}
Assume that $\Omega_{S(p,q+1)}$ and $\Omega_{S(p+1,q)}$ have a common zero at $z_0$. Then
$$
\Omega_{S(p+1,q)}'(z_0)\Omega_{S(p,q)}(z_0)=0.
$$
Since property (A.2) holds for $q$ and $p+1$ we have $\Omega_{S(p+1,q)}'(z_0)\not =0$.
But since property (A.3) holds for $q$ then $\Omega_{S(p,q)}(z_0)\not =0$. Which it is a contradiction.

We next prove that if $\Omega_{S(p,q+1)}(z_0)=0$ then $\Omega_{S(p,q+1)}'(z_0)\not=0$. ((A.2) for $q+1$)).
Indeed, assume that $\Omega_{S(p,q+1)}'(z_0)=0$. Using what we have already proved, we can conclude that then
$\Omega_{S(p+1,q)}(z_0)\not =0$. Since for $F=S(p+1,q)$, $u_F=pq$, we trivially have
$$
\Omega_{S(p,q+1)}(x)=H_{p+pq}^{\Omega_{S(p+1,q)}}(x),
$$
where $H_n^F$ denotes the exceptional Hermite polynomial of degree $n$ defined by (\ref{defhex}).  The second order differential equation (\ref{odom}) for
$F=S(p+1,q)$ can then be written in the form
\begin{align}\label{mcs}
&-\Omega_{S(p,q+1)}''(x)+2\left(x+\frac{\Omega_{S(p+1,q)}'(x)}{\Omega_{S(p+1,q)}(x)}\right)\Omega_{S(p,q+1)}'(x)\\\nonumber
&+\left(2p(q+1)-2x\frac{\Omega_{S(p+1,q)}'(x)}{\Omega_{S(p+1,q)}(x)}-\frac{\Omega_{S(p+1,q)}''(x)}{\Omega_{S(p+1,q)}(x)}\right)\Omega_{S(p,q+1)}(x)=0.
\end{align}
Since $\Omega_{S(p+1,q)}(z_0)\not =0$, we  then conclude that $\Omega_{S(p,q+1)}''(z_0)=0$.
And by deriving successively in (\ref{mcs}) and evaluating at $x=z_0$, we find $\Omega_{S(p,q+1)}^{(l)}(z_0)=0$ for all $l$. But this is a contradiction because
$\Omega_{S(p,q+1)}(x)$ is a non-null polynomial.

We finally prove that the polynomials $\Omega_{S(p,q+1)}$ and $\Omega_{S(p+1,q+1)}$ have no zeros in common ((A.3) for $q+1$).

We will use that $I(S(p,q))=S(q,p)$ where $I$ is the involution (\ref{dinv}). The invariance property (\ref{hip}) then gives
\begin{equation}\label{inva.1}
\Omega_{S(p,q)}(x)=c_{p,q}\Omega_{S(q,p)}(-ix),
\end{equation}
for certain constant $c_{p,q}\not =0$ which does not depend on $x$.

Using again the Sylvester's identity of Lemma \ref{lemS} for $M=\Omega_{S(q,p+2)}$, we get
\begin{align*}
\Omega_{S(q,p+2)}(x)\Omega_{S(q+1,p)}(x)&=\Omega_{S(q,p+1)}(x)\Omega_{S(q+1,p+1)}'(x)
\\&\qquad-\Omega_{S(q,p+1)}'(x)\Omega_{S(q+1,p+1)}(x).
\end{align*}
Using the invariance (\ref{inva.1}), this can be rewritten in the form
\begin{align*}
c_0\Omega_{S(p+2,q)}(-ix)\Omega_{S(p,q+1)}(-ix)&=c_1\Omega_{S(p+1,q)}(-ix)\Omega_{S(p+1,q+1)}(-ix)
\\&\quad \qquad+c_2\Omega_{S(p+1,q)}(-ix)\Omega_{S(p+1,q+1)}(-ix).
\end{align*}
Assume that $\Omega_{S(p,q+1)}$ and $\Omega_{S(p+1,q+1)}$ have a common zero at $z_0$. Then
$\Omega_{S(p+1,q)}(z_0)\Omega_{S(p+1,q+1)}'(z_0)=0$. Since we have already proved property (A.1) for $q+1$ we have $\Omega_{S(p+1,q)}(z_0)\not =0$.
But we have also proved property (A.2)  for $q+1$ (and $p+1$) then $\Omega_{S(p+1,q+1)}'(z_0)\not =0$. Which it is a contradiction.

\end{proof}



\end{document}